\documentclass[12pt,a4paper]{article}
\usepackage{amssymb,amsmath,amsthm}
\usepackage{graphicx}

\newcommand\cC{{\mathcal C}}

\newcommand\cF{{\mathcal F}}

\newcommand\bC{\mathbf C}
\newcommand\cG{{\mathcal G}}

\newcommand\cL{{\mathcal L}}

\newcommand\cP{{\mathcal P}}

\newcommand\cS{{\mathcal S}}

\theoremstyle{plain}
\newtheorem{theorem}{Theorem}[section]
\newtheorem{lemma}[theorem]{Lemma}

\newtheorem{proposition}[theorem]{Proposition}

\theoremstyle{definition}

\newtheorem{claim}[theorem]{Claim}

\newcommand\cref[1]{Corollary~\ref{cor:#1}}

\title{Vertex Tur\'an problems for the oriented hypercube}
\author{D\'aniel Gerbner$^{a,}$, Abhishek Methuku$^b$, D\'aniel T. Nagy$^{a,}$, \\ Bal\'azs Patk\'os$^{a,}$, M\'at\'e Vizer$^{a,}$ \\
\small $^a$ Alfr\'ed R\'enyi Institute of Mathematics, Hungarian Academy of Sciences\\
\small P.O.B. 127, Budapest H-1364, Hungary.\\
\small $^b$ Central European University, Department of Mathematics\\
\small Budapest, H-1051, N\'ador utca 9.\\
\medskip
\small \texttt{\{gerbner,nagydani,patkos\}@renyi.hu, \{abhishekmethuku,vizermate\}@gmail.com}
\medskip}

\begin{document}
\maketitle
\begin{abstract}

In this short note we consider the oriented vertex Tur\'an problem in the hypercube: for a fixed oriented graph $\overrightarrow{F}$, determine the maximum size $ex_v(\overrightarrow{F}, \overrightarrow{Q_n})$ of a subset $U$ of the vertices of the oriented hypercube $\overrightarrow{Q_n}$ such that the induced subgraph $\overrightarrow{Q_n}[U]$ does not contain any copy of  $\overrightarrow{F}$. We obtain the exact value of $ex_v(\overrightarrow{P_k}, \overrightarrow{Q_n})$ for the directed path $\overrightarrow{P_k}$, the exact value of $ex_v(\overrightarrow{V_2}, \overrightarrow{Q_n})$ for the directed cherry $\overrightarrow{V_2}$ and the asymptotic value of $ex_v(\overrightarrow{T}, \overrightarrow{Q_n})$ for any directed tree $\overrightarrow{T}$.

\end{abstract}

\section{Introduction}
One of the most studied problems in extremal combinatorics is the so-called Tur\'an problem originated in the work of Tur\'an \cite{T1941} (for a recent survey see \cite{FS2013}). A basic problem of this sort asks for the maximum possible number of edges $ex(F,G)$ in a subgraph $G'$ of a given graph $G$ that does not contain $F$ as a subgraph.


Much less attention is paid to the vertex version of this problem. This problem can be formalized as follows: what is the the maximum size $ex_v(F,G)$, of a subset $U$ of vertices of a given graph $G$ such that $G[U]$ does not contain $F$ as a subgraph.


We will consider Tur\'an type problems for the \textit{n-dimensional hypercube} $Q_n$, the graph with vertex set $V_n=\{0,1\}^n$ corresponding to subsets of an $n$-element set and edges between vertices that differ in exactly one coordinate.

Edge-Tur\'an problems in the hypercube have attracted a lot of attention. This research was initiated by Erd\H{o}s \cite{erdos1984}, who conjectured $ex(C_4,Q_n)=(1+o(1))n 2^{n-1}$, i.e., any subgraph of $Q_n$ having significantly more than half of the edges of $Q_n$ must contain a copy of $C_4$. This problem is still unsolved. Conlon \cite{conlon} showed, extending earlier results due to Chung \cite{Chu92} and F\"uredi and \"Ozkahya \cite{FurOzk09, FurOzk11}, that $ex(C_{2k},Q_n)=o(n2^n)$ for $k\neq 2,3,5$. 

Concerning the vertex Tur\'an problem in the hypercube $Q_n$, it is obvious that we can take half of the vertices of $Q_n$ such that they induce no edges.
Kostochka \cite{K} and later, independently, Johnson and Entringer \cite{JE} showed 
$ex_v(C_4,Q_n)=\max_j\{\sum_{i\not\equiv j~\mod 3}\binom{n}{i}\}$. Johnson and Talbot \cite{JT}  proved a local stability version of this result. Chung, F\"uredi, Graham, Seymour \cite{CFGS} proved that if $U$ contains more than $2^{n-1}$ vertices, then there is a vertex of degree at least $\frac{1}{2}\log n - \frac{1}{2}\log \log n +\frac{1}{2}$ in $Q_n[U]$. This shows that for any star $S_k$ with $k$ fixed, we have $ex_v(S_k,Q_n)=2^{n-1}$ for large enough $n$. Alon, Krech, and Szab\'o \cite{AKS} investigated the function $ex_v(Q_d,Q_n)$.

\smallskip
Let us note that there is a simple connection between the edge and the vertex Tur\'an problems in the hypercube. 

\begin{proposition}
$ex_v(F,Q_n)\le 2^{n-1}+\frac{ex(F,Q_n)}{n}$.

\end{proposition}

\begin{proof} If a subgraph $G$ of $Q_n$ contains more than $2^{n-1}+\frac{ex(F,Q_n)}{n}$ vertices, then it contains more than $\frac{ex(F,Q_n)}{n}$ edges in every direction, thus more than $ex(F,Q_n)$ edges altogether, hence $G$ contains a copy of $F$.

\end{proof}

This observation implies that for every tree $T$, we have $ex_v(T,Q_n)=\left(\frac{1}{2}+\mathcal{O}\left(\frac{1}{n}\right)\right)2^n$, using the well-known result from Tur\'an theory which states $ex(n,T)=O(n)$ (and so $ex(F,Q_n) = \mathcal{O}(2^n)$). Also, together with Conlon's result on the cycles mentioned earlier, we obtain $ex_v(C_k,Q_n)=\left(\frac{1}{2}+o(1)\right)2^n$ for $k\neq 2,3,5$.

\bigskip

In this paper, we consider an oriented version of this problem. There is a natural orientation of the edges of the hypercube. An edge $uv$ means that $u$ and $v$ differ in only one coordinate; if $u$ contains $1$ and $v$ contains $0$ in this coordinate, then we direct the edge from $v$ to $u$. We denote the hypercube $Q_n$ with this orientation by $\overrightarrow{Q_n}$. With this orientation it is natural to forbid oriented subgraphs. We will denote by $ex_v(\overrightarrow{F}, \overrightarrow{Q_n})$ the maximum number of vertices that an $\overrightarrow{F}$-free subgraph of $\overrightarrow{Q_n}$ can have. As vertices of the hypercube correspond to sets, instead of working with subsets of the vertices of $\overrightarrow{Q}_n$ we will consider families $\cG\subseteq 2^{[n]}$ of sets. We will say that $\cG\subseteq 2^{[n]}$ is $\overrightarrow{F}$-free if for the corresponding subset $U$ of vertices of $\overrightarrow{Q}_n$ the induced subgraph $\overrightarrow{Q}_n[U]$ is $\overrightarrow{F}$-free.

For example, there is only one orientation of $C_4$ that embeds into the hypercube, we will denote it by $\overrightarrow{C_4}$. Hence we have $ex_v(\overrightarrow{C_4}, \overrightarrow{Q_n})=ex_v(C_4,Q_n)$, which is known exactly, due to the above mentioned result of Kostochka and Johnson and Entringer.
However, there are three different orientations of $P_3$, according to how many edges go towards the middle vertex: $\overrightarrow{V_2}$ denotes the orientation with a source (i.e., $\overrightarrow{V_2}$ is the path $abc$ such that the edge $ab$ is directed from $b$ to $a$ and the edge $bc$ is directed from $b$ to $c$).
The directed path $\overrightarrow{P_k}$ is a path on $k$ vertices $v_1,\dots,v_k$ with edges going from $v_i$ to $v_{i+1}$ for every $i<k$. The \emph{height} of a directed graph is the length of a longest directed path in it.

If we consider the hypercube as the Boolean poset, then each edge of the hypercube goes between a set $A$ and a set $A\cup \{x\}$ for some $x\not\in A$. Then in $\overrightarrow{Q_n}$ the corresponding directed edge goes from $A$ to $A\cup \{x\}$. A directed acyclic graph $\overrightarrow{F}$ can be considered as a poset $F$; we will say that $F$ is the \textit{poset of} $\overrightarrow{F}$. The poset corresponding to a directed tree is said to be a \textit{tree poset}. Forbidding copies of a poset in a family of sets in this order-preserving sense has an extensive literature (see \cite{griggs2016progress} for a survey on the theory of forbidden subposets). We say $\cP\subset 2^{[n]}$ is a \textit{copy} of $P$ if there exists a bijection $f:P\rightarrow \cP$ such that $p<p'$ implies $f(p)\subset f(p')$. We say that $\cF\subset 2^{[n]}$ is \textit{P-free}, if there is no $\cP \subset \cF$ that is a copy of $P$. Observe that if $P$ is the poset of the directed acyclic graph $\overrightarrow{F}$, then any $P$-free family is $\overrightarrow{F}$-free.

The oriented version of the vertex Tur\'an problem in the hypercube corresponds to the following variant of the forbidden subposet problem.  We say $\cP\subset 2^{[n]}$ is a \emph{cover-preserving copy} of $P$ if there exists a bijection $f:P\rightarrow \cP$  such that if $p$ covers $p'$ in $P$, then $f(p)$ covers $f(p')$ in the Boolean poset. Thus it is not surprising that we can use techniques and results from the theory of forbidden subposet problems in our setting.


In this paper, we consider Vertex Tur\'an problems for directed trees. Our main result determines the asymptotic value of the vertex Tur\'an number $ex_v(\overrightarrow{T},\overrightarrow{Q_n})$ for any directed tree $\overrightarrow{T}$.

\begin{theorem}\label{dirtree}
For any directed tree $\overrightarrow{T}$ of height $h$, we have
$$ex_v(\overrightarrow{T},\overrightarrow{Q_n})=\left(\frac{h-1}{h}+o(1)\right)2^n.$$
\end{theorem}

Below we obtain the exact value of the vertex Tur\'an number for some special directed trees (namely $\overrightarrow{V_2}$ and $\overrightarrow{P_k}$).

\begin{theorem}\label{V}
$$ex_v(\overrightarrow{V_2},\overrightarrow{Q_n})=2^{n-1}+1.$$
\end{theorem}

It would be natural to consider the following generalization of $\overrightarrow{V_2}$: let $\overrightarrow{V_r}$ denote the star with $r$ leaves all edges oriented towards the leaves. Note that if one takes the elements of the $r$ highest levels of the Boolean poset and every other level below them, then the corresponding family in $\overrightarrow{Q_n}$ will be $\overrightarrow{V_r}$-free. Computing the size of this family we have $ex_v(\overrightarrow{V_r},\overrightarrow{Q_n})=2^{n-1}+\Omega(n^{r-2})$. We conjecture that $ex_v(\overrightarrow{V_r},\overrightarrow{Q_n})=2^{n-1}+\Theta(n^{r-2})$ holds for every $r\ge 3$. 

\begin{theorem}\label{path}
For any pair $k,n$ of integers with $k\le n$ we have
$$ex_v(\overrightarrow{P_k},\overrightarrow{Q_n})=\max_{j\in [k]}\left\{\sum_{i\not\equiv j ~\text{mod}\ k}\binom{n}{i}\right\}.$$
\end{theorem}

\section{Proofs}

\subsection{Proof of Theorem \ref{dirtree}}
We follow the lines of a proof of Bukh \cite{Buk09} that shows that if $T$ is a tree poset with $h(T)=k$ and $\cF\subseteq 2^{[n]}$ is a $T$-free family of sets, then $|\cF|\le (k-1+O(\frac{1}{n}))\binom{n}{\lfloor n/2\rfloor}$ holds.
The proof of this theorem consists of several lemmas. Some of them we will state and use in their original form, some others we will state and prove in a slightly altered way so that we can apply them in our setting. First we need several definitions. For a family $\cF\subseteq 2^{[n]}$, its \textit{Lubell-function} $$\lambda_n(\cF)=\sum_{F\in \cF}\frac{1}{\binom{n}{|F|}}=\frac{1}{n!}\sum_{F\in \cF}|F|!(n-|F|)!$$ is the average number of sets in $\cF$ that a maximal chain $\cC$ in $2^{[n]}$ contains. A poset $P$ is called \textit{saturated} if all its maximal chains have length $h(P)$. For any poset $T$ its \textit{opposite poset} $T'$ consists of the same elements as $T$ with $t\le_{T'} t'$ if and only if $t'\le_T t$. For a family $\cF\subseteq 2^{[n]}$ of sets, its complement family is $\overline{\cF}=\{[n]\setminus F:F\in\cF\}$. Clearly, $\cF$ contains a copy of $P$ if and only if $\overline{\cF}$ contains a copy of $P'$ and $\lambda_n(\cF)=\lambda_n(\overline{\cF})$.


\begin{lemma}[Bukh \cite{Buk09}]\label{inducedsaturated}
Every tree poset $T$ is an induced subposet of a saturated tree poset $T'$ with $h(T)=h(T')$.
\end{lemma}

An \emph{interval} in a poset $P$ is a set of the form $[x,y] = \{z \in  P : x \le z \le y\}.$

\begin{lemma}[Bukh \cite{Buk09}]\label{saturatedgrow}
If $T$ is a saturated tree poset that is not a chain, then there exists $t \in T$ that is a leaf in $H(T)$ and there exists an interval $I\subset T$ containing $t$ such that $|I|<h(T)$ holds, and $T\setminus I$ is a saturated tree poset with $h(T)=h(T\setminus I)$.
\end{lemma}

From now on we fix a tree poset $T$ and we denote its height by $k$. We say that a chain in $2^{[n]}$ is \textit{fat} if it contains $k$ members of $\cF$.

\begin{lemma}\label{bukh1}
If $\cF\subseteq \bigcup_{j=i}^{i+k-1} \binom{[n]}{j}$ is a family with $\lambda(\cF)\ge (k-1+\varepsilon)$, then there are at least $(\varepsilon/k)n!$ fat chains.
\end{lemma}

\begin{proof}
Let $C_i$ denote the number of maximal chains that contain exactly $i$ sets from $\cF$. As $\cF\subseteq \bigcup_{j=i}^{i+k-1} \binom{[n]}{j}$, we have $C_i=0$ for all $i>k$. Then
counting the number of pairs $(F, \cC)$ with $\cC$ being a maximal chain and $F \in \cF \cap \cC$, 
in two different ways, we obtain
\[
\sum_{i=0}^niC_i=\lambda(\cF)n!\ge  (k-1+\varepsilon)n!.
\]
This, and $\sum_iC_i=n!$ imply
\[
kC_k=\sum_{i\ge k}iC_i\ge  \sum_{i=0}^niC_i-(k-1)\sum_{i<k}C_i\ge \varepsilon n!.
\]
Therefore the number of fat chains in $\cF$ is $C_k\ge (\varepsilon/k)n!$.

\end{proof}

\begin{lemma}\label{bukhmain}
Let $T$ be a saturated tree poset of height $k$. Suppose $\cF\subseteq \cup_{j=i}^{i+k-1} \binom{[n]}{j}$ is a family with $n/4\le i \le 3n/4$. Moreover, suppose $\cL$ is a
family of 
fat chains with
$|\cL| >\frac{4\binom{|T|+1}{2}}{n}n!$.
Then there is a copy of $T$ in $\cF$ that contains only sets that are contained in some fat chain in $\cL$.
\end{lemma}

\begin{proof}
We proceed by induction on $|T|$. If $T$ is a chain, then the $k$ sets in any element of $\cL$
form a copy of $T$. In particular, it gives the base case of the induction. So suppose $T$ is not a chain. Then applying Lemma \ref{saturatedgrow}, there exists a leaf $t$ in $T$ and interval $I \subseteq T$ containing $t$ such that $h(T\setminus I)=k$ and $T\setminus I$ is a saturated tree poset. Our aim is to use induction to obtain a copy of $T\setminus I$ in $\cF$ that can be extended to a copy of $T$. Finding a copy of $T\setminus I$ is immediate, but in order to be able to extend it, we need a copy satisfying some additional properties, described later.

By passing to the opposite poset $T'$ of $T$ and considering $\overline{\cF}$, we may assume that $t$ is a minimal element of $T$. There exists a maximal chain $C$ in $T$ that contains $I$, and we have $|C|=k$ as $T$ is saturated. Then $s:=|C\setminus I|=k-|I|\ge 1$.

We need several definitions. Let $F_1\supset F_2\supset \dots \supset F_s$ be a chain with $|F_j|=i+k-j$ for $j=1,\dots,s$. Then this chain is
a \textit{bottleneck} if there exists a family $\cS\subset \cF$ with $|\cS|<|T|$ such that for every fat chain  $F_1 \supset F_2\supset \dots\supset F_s\supset F_{s+1}\supset\dots\supset F_k$ in $\cL$ we have $\cS\cap \{F_{s+1},\dots, F_k\}\neq \emptyset$. Such an $\cS$ is a \textit{witness} to the fact that $F_1,\dots, F_s$ is a bottleneck (and we assume all sets of the witness are contained in $F_s$). We say that a fat chain is \textit{bad} if its top $s$ sets form a bottleneck. A fat chain is \textit{good} if it is not bad. Observe that if there is a copy $\cF_{T\setminus I}$ of $T\setminus I$ consisting of sets of good fat chains, then we can extend $\cF_{T\setminus I}$ to a copy of $T$. Indeed, as the sets $F'_1,\dots,F'_s$ representing $C\setminus I$ in $\cF_{T\setminus I}$ do not form a bottleneck and $|\cF_{T\setminus I}|<|T|$, there must be a good fat chain $F'_1\supset \dots\supset F'_s\supset F'_{s+1}\supset \dots\supset F'_k$ such that $F'_{s+1},\dots,F'_k\notin \cF_{T\setminus I}$, therefore $\cF_{T\setminus I}\cup \{F'_{s+1},\dots, F'_k\}$ is a copy of $T$. Therefore all we need to prove is that there are enough good fat chains to obtain a copy of $T\setminus I$ by induction.


Let us bound the number of bad fat chains. If $|\cC\cap \cF|<s$, then $\cC$ cannot be bad. We partition maximal chains in $2^{[n]}$ according to their $s$th largest set $F_{s}$ from $\cF$. As the top $s$ sets must form a bottleneck, there is a witness $\cS$ to this fact. This means that if $\cC$ is bad, then $\cC$ must meet $\cS$ whose elements are all contained in $F_{s}$. But as $|\cS|<|T|$ and all sets of $2^{F_{s}}\cap \cF$ have size between $n/4$ and $3n/4$, the proportion of those chains that do meet $\cS$ is at most $4|T|/n$ (any proper non-empty subset of $F_{S}$ is contained in at most $1/|F_{s}|$ proportion of chains going through $F_{s}$). This holds independently of the choice of $F_{s}$, thus the number of bad fat chains is at most $\frac{4|T|}{n}n!$. 
So the number of good fat chains is at least
\[
|\cL|-\frac{4|T|}{n}n! \ge \frac{4(\binom{|T|+1}{2}-|T|)}{n}n!=\frac{4\binom{|T|}{2}}{n}n!.
\]
As $|T\setminus I|<|T|$, the induction hypothesis implies the existence of a copy of $T\setminus I$ among the 
sets contained in good fat chains, as required.
\end{proof}

\vskip 0.3truecm

The next lemma essentially states that if a a $T$-free family is contained in the union of $k$ consecutive levels,  then its size is asymptotically at most the cardinality of the $k-1$ largest levels. Formally, let $b(i)=b_{k,n}(i)=\max\{\binom{n}{j}:i\le j \le i+k-1\}$. So if $i\le n/2-k+1$, then $b(i)=\binom{n}{i+k-1}$, if $i\ge n/2$, then $b(i)=\binom{n}{i}$, while if $n/2-k+1<i<n/2$, then $b(i)=\binom{n}{\lfloor n/2\rfloor}$.

\begin{lemma}\label{bukhuj} If $T$ is a tree poset of height $k$, then there exists $n_0$ such that for $n>n_0$, $n/4\le i\le 3n/4-k$ any $\cF\subset \bigcup_{j=i}^{i+k-1} \binom{[n]}{j}$ of size at least $\left(k-1+\frac{k4|T|^2}{n}\right)b(i)$ contains a copy of $T$.
\end{lemma}

\begin{proof} By Lemma \ref{inducedsaturated} we may suppose that $T$ is a saturated tree poset. Assume $\cF\subseteq \bigcup_{j=i}^{i+k-1} \binom{[n]}{j}$ is a $T$-free family that contains at least $\left(k-1+\frac{k4|T|^2}{n}\right)b(i)$ sets. Then $\cF\subseteq \bigcup_{j=i}^{i+k-1} \binom{[n]}{j}$ implies that $\lambda_n(\cF)\ge k-1+\frac{k4|T|^2}{n}$.

Let $\varepsilon=4k|T|^2/n$. Then we can apply Lemma \ref{bukh1} to find $4|T|^2n!/n$ fat chains. Then we can apply Lemma \ref{bukhmain} with $k=h(T)$ to obtain a copy of $T$ in $\cF$, contradicting the $T$-free property of $\cF$.
\end{proof}

With  Lemma \ref{bukhuj} in hand, we can now prove Theorem \ref{dirtree}. Let us consider a $\overrightarrow{T}$-free family $\cF$. Let $T$ be the poset of $\overrightarrow{T}$ and let $T^*$ be the saturated poset containing $T$ with $h(T)=h(T^*)=k$ - guaranteed by Lemma \ref{inducedsaturated}. For any integer $0\le i \le n-k+1$, let $\cF_i=\{F\in \cF: i\le |F|\le i+k-1\}$. Observe that the $\overrightarrow{T}$-free property of $\cF$ implies that $\cF_i$ is $T^*$-free for every $i$. Note that every $F\in \cF$ belongs to exactly $k$ families $\cF_i$ unless $|F|<k-1$ or $|F|>n-k+1$. It is well-known that $\left|\binom{[n]}{\le n/4}\cup \binom{[n]}{\ge 3n/4}\right|=o\left(\frac{1}{n}2^{n}\right)$, therefore using Lemma \ref{bukhuj} we obtain
\[
k|\cF|-o\left(\frac{1}{n}2^n\right)\le \sum_{i=n/4}^{3n/4}|\cF_i|\le \left(k-1+\frac{k4|T|^2}{n}\right)\sum_{i=n/4}^{3n/4}b(i)\le \left(k-1+\frac{k4|T|^2}{n}\right)\left(2^n+k\binom{n}{\lfloor n/2\rfloor}\right).
\]
After rearranging, we get $|\cF|\le\left(\frac{k-1}{k}+o(1)\right)2^n$.

\subsection{Proof of Theorem \ref{V}}
To prove the lower bound, we show a $\overrightarrow{V_2}$-free family in $\overrightarrow{Q_n}$ of size $2^{n-1}+1$. Simply take every second level in the hypercube starting from the $(n-1)$st level and also take the vertex corresponding to $[n]$.

We prove the upper bound by induction on $n$ (it is easy to see the base case $n=2$). We will need the following simple claim.

\begin{claim}\label{vclaim}
Let $\cF\subset 2^{[n]}$ is a maximal $\overrightarrow{V_2}$-free family, then $\cF$ contains the set $[n]$ and at least one set of size $n-1$.
\end{claim}

\begin{proof}[Proof of Claim]
First note that $[n]$ can be added to any $\overrightarrow{V_2}$-free family as there is only one subset of $[n]$ of size $n$. Also, if $\cF$ does not contain any set of size $n-1$, then one such set $S$ can be added to $\cF$. Indeed, if we add $S$, no copy of $\overrightarrow{V_2}$ having sets of size $n-1$ and $n$ will be created because $[n]$ is the only set of size $n$ in $\cF\cup \{S\}$. Furthermore, no copy of $\overrightarrow{V_2}$ having sets of size $n-2$ and $n-1$ will be created as $S$ is the only set of size $n-1$ in $\cF\cup \{S\}$.
\end{proof}

Now we are ready to prove Theorem \ref{V}. Let $\cF\subset 2^{[n]}$ be a $\overrightarrow{V_2}$-free family. For some $x\in [n]$, define
$$\cF_x^-=\{F~|~F\in\cF,~x\not\in F\}~~~\text{and}~~~\cF_x^+=\{F\backslash\{x\}~|~F\in\cF,~x\in F\}.$$

Then $\cF_x^-,\cF_x^+\subset 2^{[n]\backslash\{x\}}$ and they are also $\overrightarrow{V_2}$-free. By induction, we have
$$|\cF|=|\cF_x^-|+|\cF_x^+|\le 2^{n-2}+1+2^{n-2}+1=2^{n-1}+2.$$

Assume that $|\cF|=2^{n-1}+2.$ Then $|\cF_x^-|=|\cF_x^+|=2^{n-2}+1$ must hold for all $x\in [n]$. By Claim \ref{vclaim}, $|\cF_x^-|=2^{n-2}+1$ implies that $[n]\backslash\{x\}$ and at least one set of size $n-2$ are in $\cF$. This holds for all $x\in [n]$, so all sets of size $n-1$, and at least one set of size $n-2$ are in $\cF$. However, these would form a forbidden $\overrightarrow{V_2}$ in $\cF$, contradicting our original assumption on $\cF$. This proves that $|\cF|\le 2^{n-1}+1$.

\subsection{Proof of Theorem \ref{path}}
Let $U$ be a set of vertices in $Q_n$ such that the subgraph of $Q_n$ induced by $U$ (i.e., $Q_n[U]$) is $\overrightarrow{P_k}$-free. Let $\cF \subset 2^{[n]}$ be a family of subsets corresponding to $U$.

First, we will introduce a weight function. For every $F\in\cF$, let $w(F)=\binom{n}{|F|}$. For a maximal chain $\cC$, let $w(\cC)=\sum_{F\in \cC\cap \cF}w(F)$ denote the weight of $\cC$. Let $\bC_n$ denote the set of all maximal chains in $[n]$. Then
$$\frac{1}{n!}\sum_{\cC\in\bC_n} w(\cC)=\frac{1}{n!}\sum_{\cC\in\bC_n}\sum_{F\in \cC\cap \cF}w(F)=\frac{1}{n!}\sum_{F\in \cF} |F|!(n-|F|)!w(F)=|\cF|.$$

This means that the average of the weights of the full chains equals the size of $\cF$. It means that if we find an upper bound that is valid for the weight of any chain, then this will be an upper bound on $|\cF|$ too.

Our assumption that there is no $\overrightarrow{P_k}$ means that there are no $k$ neighboring members of $\cF$ in a chain. For a given chain $\cC$, let $a_1, a_2,\dots a_t$ denote the sizes of those elements of $\cC$ that are not in $\cF$. Then $0\le a_1<a_2<\dots <a_t\le n$, $a_1\le k-1$, $n-k+1\le a_t$ and $a_{i+1}-a_{i}\le k$ for all $i=1,2,\dots t-1$. The weight of the chain $\cC$ is
$$w(\cC)=2^n-\sum_{i=1}^t \binom{n}{a_i}.$$

We claim that this is maximized when the numbers $\{a_1, a_2,\dots a_t\}$ are all the numbers between 0 and $n$ that give the same residue when divided by $k$.

Assume that $w(\cC)$ is maximized by a different kind of set $\{a_1, a_2,\dots a_t\}$. Then there is an index $i$ such that $a_{i+1}-a_{i}<k$.

If $a_i\le \frac{n}{2}$ then we can decrease the numbers $\{a_1, a_2,\dots a_i\}$ by 1. (If $a_1$ becomes -1 then we simply remove that number.) The resulting set of numbers will still satisfy the conditions and $w(\cC)$ increases. Otherwise, $a_{i+1}> \frac{n}{2}$ must hold. Similarly, we can increase the numbers $\{a_{i+1}, a_{i+2},\dots a_n\}$ by 1 to achieve the same result. We proved that
$$w(\cC)\le 2^n-\min_{j\in [k]}\sum_{i\equiv j ~\text{mod}\ k}\binom{n}{i} =\max_{j\in [k]}\left\{\sum_{i\not\equiv j ~\text{mod}\ k}\binom{n}{i}\right\}$$
holds for any full chain $\cC$. Therefore the same upper bound holds for $|\cF|$ as well.

\subsection*{Acknowledgement}

Research of D. Gerbner was supported by the J\'anos Bolyai Research Fellowship of the Hungarian Academy of Sciences and the National Research, Development and Innovation Office -- NKFIH under the grant K 116769.

Research of A. Methuku was supported by the Hungarian Academy of Sciences and the National Research, Development and Innovation Office -- NKFIH under the grant K 116769.

Research of D.T. Nagy was supported by the \'{U}NKP-17-3 New National Excellence Program of the Ministry of Human Capacities and by National Research, Development and Innovation Office - NKFIH under the grant K 116769.

Research of B. Patk\'os was supported by the National Research, Development and Innovation Office -- NKFIH under the grants SNN 116095 and K 116769.

Research of M. Vizer was supported by the National Research, Development and Innovation Office -- NKFIH under the grant SNN 116095.

All of the authors were funded by the Taiwanese-Hungarian Mobility Program of the Hungarian Academy of Sciences.

\bibliography{vtbib}
\bibliographystyle{abbrv}

\end{document}